\documentclass[12pt]{amsart}

\usepackage{accents}
\usepackage{appendix}
\usepackage{amsfonts}
\usepackage{amsmath}
\usepackage{amssymb}	
\usepackage{amsthm}
\usepackage{array,booktabs,multirow}
\usepackage{braket}
\usepackage{cite}
\usepackage{centernot}
\usepackage{dsfont}
\usepackage{mathalfa}
\usepackage[shortlabels]{enumitem} 
\usepackage{etoolbox}
\usepackage{float}
\usepackage[hang, flushmargin]{footmisc}
\usepackage{latexsym}
\usepackage{lipsum}
\usepackage{needspace}
\usepackage{tikz}
\usepackage{hyperref}
\usetikzlibrary{matrix,arrows}  

\usepackage{tikz-cd}

\theoremstyle{plain}
\newtheorem{thm}{Theorem}[section]

\newtheorem{lem}[thm]{Lemma}
\newtheorem{prop}[thm]{Proposition}

\theoremstyle{definition}

\newtheorem{defn}[thm]{Definition}
\newtheorem{exam}[thm]{Example}

\theoremstyle{remark}
\def\BI{\mathfrak{BI}}

\setlist[enumerate,1]{leftmargin=2em}

\def\N{\mathbb N}
\def\F{\mathbb F}
\def\Z{\mathbb Z}

\def\e{\varepsilon}

\title[Finite-dimensional irreducible $\BI$-modules at ${\rm\,}\F=0$]{Finite-dimensional irreducible modules of the Bannai--Ito algebra at characteristic zero}

\author{Hau-Wen Huang}
\address{
Hau-Wen Huang\\
Department of Mathematics\\
National Central University\\
Chung-Li 32001 Taiwan
}
\email{hauwenh@math.ncu.edu.tw}

\begin{document}

\begin{abstract}
Assume that $\F$ is an algebraically closed with characteristic $0$. A central extension $\BI$ of the Bannai--Ito algebras is a unital associative $\F$-algebra generated by $X,Y,Z$ and the relations assert that each of 
\begin{gather*}
\{X,Y\}-Z, 
\qquad 
\{Y,Z\}-X,
\qquad 
\{Z,X\}-Y
\end{gather*}
is central in $\BI$. In this paper we classify the finite-dimensional irreducible $\BI$-modules up to isomorphism. As we will see the elements $X,Y,Z$ are not always diagonalizable on finite-dimensional irreducible $\BI$-modules.
\end{abstract}

\maketitle


\section{Introduction}

Throughout this paper, we adopt the following notations: 
Let $\F$ denote a field and let ${\rm char\,}\F$ denote the characteristic of $\F$. 
Let $\Z$ denote the ring of integers. Let $\N$ denote the set of nonnegative integers. 
Recall that the anticommutator $\{X,Y\}$ of two elements $X,Y$ in an algebra is defined by $\{X,Y\}=XY+YX$.

In this paper we consider a central extension of the Bannai--Ito algebras called the universal Bannai--Ito algebra and denoted by $\BI$. The algebra $\BI$ is a unital associative $\F$-algebra defined by generators and relations. The generators are $X,Y,Z$ and the relations state each of 
\begin{gather*}
\{X,Y\}-Z, 
\qquad 
\{Y,Z\}-X,
\qquad 
\{Z,X\}-Y
\end{gather*}
commutes with $X,Y,Z$. The concept of central extensions comes from \cite{uaw2011}. 
The Bannai--Ito algebras are the case $q=-1$ of the Askey--Wilson algebras \cite{lp&awrelation,hidden_sym} and the corresponding orthogonal polynomials were first known in \cite{BannaiIto1984}. In \cite{tvz2012}, the Bannai--Ito algebras 
were reintroduced to connect the Dunkl shift operators and the Bannai--Ito polynomials. Recently, the Bannai--Ito algebras and their representation theory have been explored in many other subjects such as the additive DAHA of type $(C_1^\vee,C_1)$ \cite{BI&NW2016,Huang:R<BImodules}, the Lie superalgebra $\mathfrak{osp}(1|2)$ \cite{BI&osp2018,Vinet2019}, the Racah algebras \cite{R&BI2015,Huang:R<BI} and the Brauer algebra \cite{Vinet2019}. The realizations of the Bannai--Ito algebras via  Dunkl harmonic analysis on the two-sphere were displayed in \cite{BI2014-2,BI2015,BI2016}. For generalizations of the Bannai--Ito algebras please refer to \cite{BI2018,HBI2016,qBI2019,qBI2018}.

Assume that $\F$ is algebraically closed with ${\rm char\,}\F=0$. It was falsely claimed in \cite[Lemma 5.9]{BImodule2016} that $X,Y,Z$ are diagonalizable on each finite-dimensional irreducible $\BI$-module and the mistake was used to classify even-dimensional irreducible $\BI$-modules \cite[Theorem 6.15]{BImodule2016} and odd-dimensional irreducible $\BI$-modules \cite[Theorem 7.5]{BImodule2016}. We display the following two examples to pinpoint the failure of \cite[Theorem 6.15]{BImodule2016} and \cite[Theorem 7.5]{BImodule2016}, respectively:

\allowdisplaybreaks
\begin{exam}\label{exam:E}
It is routine to verify that there exists a four-dimensional $\BI$-module $E$ that has an $\F$-basis $\{u_i\}_{i=0}^3$ with respect to which the matrices representing $X,Y,Z$ are  
\begin{align*}
\begin{pmatrix}
-\frac{1}{2} &0 &0 &0
\\
1 &-\frac{1}{2} &0 &0
\\
0 &1 &\frac{3}{2} &0
\\
0 &0 &1 &-\frac{5}{2}
\end{pmatrix},
\quad 
\begin{pmatrix}
-\frac{3}{2} &1 &0 &0
\\
0 &\frac{1}{2} &4 &0
\\
0 &0 &\frac{1}{2} &-3
\\
0 &0 &0 &-\frac{3}{2}
\end{pmatrix},
\quad 
\begin{pmatrix}
-\frac{3}{2} &-1 &0 &0
\\
-1 &\frac{1}{2} &4 &0
\\
0 &1 &-\frac{3}{2} &3
\\
0 &0 &-1 &\frac{1}{2}
\end{pmatrix},
\end{align*}
respectively.  More precisely 
\begin{gather*}
\{X,Y\}=Z+4,
\qquad
\{Y,Z\}=X+4,
\qquad 
\{Z,X\}=Y+2
\end{gather*}
on the $\BI$-module $E$. 
It is straightforward to check that the minimal polynomials of $X,Y,Z$ on $E$ are 
\begin{align*}
&\left(
x-\frac{3}{2}
\right)
\left(x+\frac{1}{2}
\right)^2
\left(
x+\frac{5}{2}
\right),
\\
&
\left(
x-\frac{1}{2}
\right)^2
\left(x+\frac{3}{2}
\right)^2,
\\
&\left(
x-\frac{3}{2}
\right)
\left(x+\frac{1}{2}
\right)^2
\left(
x+\frac{5}{2}
\right),
\end{align*}
respectively. Therefore none of $X,Y,Z$ is diagonalizable on $E$.
To see the irreducibility of $E$,  we suppose that $W$ is any nonzero $\BI$-submodule of $E$ and show that $W=E$.  
The element $Y$ has exactly two eigenvalues $-\frac{3}{2}$ and $\frac{1}{2}$ in $E$ and the $-\frac{3}{2}$- and $\frac{1}{2}$-eigenspaces of $Y$ in $E$ are of dimension $1$ spanned by 
$$
u_0,
\qquad 
u_0+2u_1
$$ 
respectively. 
Since $W$ is nonzero $-\frac{3}{2}$ or $\frac{1}{2}$ is an eigenvalue of $Y$ in $W$. Therefore $W$ contains $u_0$ or $u_0+2u_1$. If $u_0+2u_1\in W$ then 
$$
u_0=\frac{1}{3}(X-Z)(u_0+2u_1)\in W.
$$
By these comments $u_0\in W$. The $\BI$-module $E$ is generated by $u_0$, so $W=E$, a counterexample to \cite[Theorem 6.15]{BImodule2016}.
\end{exam}

\begin{exam}\label{exam:O}
It is routine to verify that there exists a five-dimensional $\BI$-module $O$ that has an $\F$-basis with respect to which the matrices representing $X,Y,Z$ are  
\begin{align*}
\begin{pmatrix}
-\frac{1}{2} &0 &0 &0 &0
\\
1 &-\frac{1}{2} &0 &0 &0 
\\
0 &1 &\frac{3}{2} &0 &0 
\\
0 &0 &1 &-\frac{5}{2} &0 
\\
0 &0 &0 &1 &\frac{7}{2} 
\end{pmatrix},
\quad 
\begin{pmatrix}
-\frac{3}{2} &4 &0 &0 &0
\\
0 &\frac{1}{2} &-2 &0 &0 
\\
0 &0 &\frac{1}{2} &6 &0 
\\
0 &0 &0 &-\frac{3}{2} &-12
\\
0 &0 &0 &0 &\frac{5}{2} 
\end{pmatrix},
\quad 
\begin{pmatrix}
\frac{3}{2} &-4 &0 &0 &0
\\
-1 &-\frac{5}{2} &-2 &0 &0 
\\
0 &1 &\frac{3}{2} &-6 &0 
\\
0 &0 &-1 &-\frac{5}{2} &-12
\\
0 &0 &0 &1 &\frac{3}{2} 
\end{pmatrix},
\end{align*}
respectively. More precisely
$$
\{X,Y\}=Z+4,
\qquad 
\{Y,Z\}=X-8,
\qquad 
\{Z,X\}=Y-4
$$
on the $\BI$-module $O$. It is routine to check that the minimal polynomials of $X,Y,Z$ on $O$ are 
\begin{align*}
&
\left(
x-\frac{7}{2}
\right)
\left(
x-\frac{3}{2}
\right)
\left(
x+\frac{1}{2}
\right)^2
\left(x+\frac{5}{2}
\right)
,
\\
&
\left(
x-\frac{5}{2}
\right)
\left(
x-\frac{1}{2}
\right)^2
\left(x+\frac{3}{2}
\right)^2,
\\
&\left(
x-\frac{3}{2}
\right)^2
\left(x+\frac{1}{2}
\right)^2
\left(
x+\frac{5}{2}
\right),
\end{align*}
respectively. Therefore none of $X,Y,Z$ is diagonalizable on $O$.
Similar to Example \ref{exam:E}, to see the irreducibility of $O$ we suppose that $W$ is any nonzero $\BI$-submodule of $O$ and show that $W=O$.  
The element $Y$ has exactly three eigenvalues $-\frac{3}{2},\frac{1}{2},\frac{5}{2}$ in $O$ and the $-\frac{3}{2}$-, $\frac{1}{2}$-, $\frac{5}{2}$-eigenspaces of $Y$ in $O$ are of dimension $1$ spanned by 
$$
u_0,
\qquad 
u_0+\frac{1}{2}u_1,
\qquad 
u_0+u_1-u_2-\frac{1}{3} u_3+\frac{1}{9} u_4
$$ 
respectively. 
Since $W$ is nonzero $-\frac{3}{2},\frac{1}{2}$ or $\frac{5}{2}$ is an eigenvalue of $Y$ in $W$. Therefore $W$ contains $u_0,u_0+\frac{1}{2}u_1$ or $u_0+u_1-u_2-\frac{1}{3} u_3+\frac{1}{9} u_4$. If $u_0+\frac{1}{2}u_1\in W$ then 
$$
u_0=\frac{1}{6}(Z-X+6)(u_0+\frac{1}{2}u_1)\in W.
$$
If $u_0+u_1-u_2-\frac{1}{3} u_3+\frac{1}{9} u_4\in W$ then 
$$
u_0=-\frac{1}{4}(3X+Z)(u_0+u_1-u_2-\frac{1}{3} u_3+\frac{1}{9} u_4)\in W.
$$
By these comments $u_0\in W$. The $\BI$-module $O$ is generated by $u_0$, so $W=O$, a counterexample to \cite[Theorem 7.5]{BImodule2016}.
\end{exam}

\noindent The main result of this paper is to give a complete classification of finite-dimensional irreducible $\BI$-modules, which answers the first open question listed in \cite[\S 11]{BI2015-2}. The proof idea originates from \cite{Huang:2015}. 
We mentioned earlier that \cite{BImodule2016} contains a mistake. Additionally, the same mistake was made in the case of the Racah algebras \cite{Rmodule2019}. S. Bockting--Conrad and the present author write a paper \cite{SH:2019-1} to correct it in greater detail. 
The results of \cite{R&BI2015,Huang:R<BI} reveal that the universal Racah algebra $\Re$ is isomorphic to an $\F$-subalgebra of $\BI$.
As an application of \cite{SH:2019-1} and our result, the lattices of $\Re$-submodules of finite-dimensional irreducible $\BI$-modules are classified in \cite{Huang:R<BImodules}.

The paper is organized as follows: In \S\ref{s:result} we state the classification of irreducible $\BI$-modules that have even and odd dimensions in Theorem \ref{thm:even} and Theorem \ref{thm:odd}, respectively. In \S\ref{s:Verma} we introduce our main tool, an infinite-dimensional $\BI$-module. In \S\ref{s:irr} we establish the necessary and sufficient conditions for the irreducibility of even-dimensional $\BI$-modules. In \S\ref{s:iso_class} we study the isomorphism classes of even-dimensional $\BI$-modules. In \S\ref{s:classification_even} we give a proof for Theorem \ref{thm:even}. Theorem \ref{thm:odd} follows by a similar argument.

\section{Statement of results}\label{s:result}

\begin{defn}\label{defn:BI}
The {\it universal Bannai--Ito algebra} $\BI$ is a unital associative $\F$-algebra generated by $X,Y,Z$ and the relations assert that each of the following elements commutes with $X,Y,Z$:
\begin{align}
&\{X,Y\}-Z, 
\label{kappa}
\\
&\{Y,Z\}-X,
\label{lambda}
\\ 
&\{Z,X\}-Y.
\label{mu}
\end{align}
\end{defn}

For notational convenience, we define $\kappa,\lambda,\mu$ to be the central elements (\ref{kappa})--(\ref{mu}) of $\BI$, respectively.

\begin{lem}\label{lem:XYkappa}
The algebra $\BI$ is generated by $X,Y,\kappa$.
\end{lem}
\begin{proof}
By (\ref{kappa}) the element $Z$ can be expressed in term of $X,Y,\kappa$. Combined with Definition \ref{defn:BI} the lemma follows.
\end{proof}

\begin{lem}\label{lem:BIpresentation}
The algebra $\BI$ has a presentation given by generators $X,Y,\kappa,\lambda,\mu$ and relations
\allowdisplaybreaks
\begin{gather}
Y^2 X+2YXY+X Y^2-X=2\kappa Y+\lambda,
\label{YYX}\\
X^2 Y+2 XYX+Y X^2-Y=2\kappa X+\mu,
\label{XXY}\\
\lambda X=X\lambda,
\quad 
\mu X=X\mu,
\quad
\kappa X=X\kappa,
\label{r:BI1}
\\
\lambda Y=Y\lambda,
\quad 
\mu Y=Y\mu,
\quad
\kappa Y=Y\kappa,
\label{r:BI2}
\\
\lambda \kappa=\kappa\lambda,
\qquad 
\mu \kappa=\kappa\mu.
\label{r:BI3}
\end{gather}
\end{lem}
\begin{proof}
Rewrite (\ref{kappa}) as 
\begin{gather}\label{e:Z}
Z=\{X,Y\}-\kappa.
\end{gather}
Relations (\ref{YYX}) and (\ref{XXY}) follow by substituting (\ref{e:Z}) into (\ref{lambda}) and (\ref{mu}), respectively. Relations (\ref{r:BI1})--(\ref{r:BI3}) reformulate the commutation of $\kappa,\lambda,\mu$ with $X,Y,Z$. 
\end{proof}

\begin{prop}\label{prop:Ed} 
For any scalars $a,b,c\in \F$ and any odd integer $d\geq 1$, there exists a $(d+1)$-dimensional $\BI$-module $E_d(a,b,c)$ satisfying the following conditions: 
\begin{enumerate}
\item There exists an $\mathbb F$-basis $\{v_i\}_{i=0}^d$ for $E_d(a,b,c)$ with respect to which the matrices representing $X$ and $Y$ are 
$$
\begin{pmatrix}
\theta_0 & & &  &{\bf 0}
\\ 
1 &\theta_1 
\\
&1 &\theta_2
 \\
& &\ddots &\ddots
 \\
{\bf 0} & & &1 &\theta_d
\end{pmatrix},
\qquad 
\begin{pmatrix}
\theta_0^* &\varphi_1 &  & &{\bf 0}
\\ 
 &\theta_1^* &\varphi_2
\\
 &  &\theta_2^* &\ddots
 \\
 & & &\ddots &\varphi_d
 \\
{\bf 0}  & & & &\theta_d^*
\end{pmatrix}
$$
respectively, where 
\allowdisplaybreaks
\begin{align*}
\theta_i &=\frac{(-1)^i(2a-d+2i)}{2}  \qquad (0\leq i\leq d),
\\
\theta^*_i &=\frac{(-1)^i(2b-d+2i)}{2} \qquad (0\leq i\leq d),
\\
\varphi_i &=
\left\{
\begin{array}{ll}
\displaystyle{i (d-i+1)}
\qquad &\hbox{if $i$ is even},
\\
\displaystyle{c^2-\frac{(2a+2b-d+2i-1)^2}{4}}
\qquad &\hbox{if $i$ is odd}
\end{array}
\right.
\qquad (1\leq i\leq d).
\end{align*}

\item The elements $\kappa,\lambda,\mu$ act on $E_d(a,b,c)$ as scalar multiplication by
\begin{align*}
c^2-a^2-b^2+\frac{(d+1)^2}{4},
\\
a^2-b^2-c^2+\frac{(d+1)^2}{4},
\\
b^2-c^2-a^2+\frac{(d+1)^2}{4},
\end{align*}
respectively.
\end{enumerate}
\end{prop}
\begin{proof}
It is straightforward to verify the proposition using Lemma \ref{lem:BIpresentation}.
\end{proof}

Recall that $\{\pm 1\}$ is a group under multiplication and the group $\{\pm 1\}^2$ is isomorphic to the Klein $4$-group. 
Observe that there exists a unique $\{\pm 1\}^2$-action on $\BI$ such that each $(\e,\e')\in \{\pm 1\}^2$ acts on $\BI$ as an $\F$-algebra automorphism in the following way:

\begin{table}[H]
\centering
\extrarowheight=3pt
\begin{tabular}{c|rrr}
$u$  &$X$ &$Y$ &$Z$ 
\\

\midrule[1pt]

$u^{(1,1)}$ &$X$ &$Y$ &$Z$ 
\\
$u^{(1,-1)}$ &$X$  &$-Y$ &$-Z$
\\
$u^{(-1,1)}$ &$-X$  &$Y$ &$-Z$
\\
$u^{(-1,-1)}$ &$-X$  &$-Y$ &$Z$
\end{tabular}
\caption{The $\{\pm 1\}^2$-action on $\BI$}\label{pm1-action}
\end{table}
\noindent Let $V$ denote a $\BI$-module. For any $(\e,\e')\in\{\pm 1\}^2$, we define 
$$
V^{(\e,\e')}
$$ 
to be the $\BI$-module obtained by twisting the $\BI$-module $V$ via $(\e,\e')$. 
The classification of even-dimensional irreducible $\BI$-modules is as follows:

\begin{thm}\label{thm:even}
Assume that $\F$ is algebraically closed with ${\rm char\,}\F=0$. 
Let $d\geq 1$ denote an odd integer. 
Let $\mathbf{EM}_d$ denote the set of all isomorphism classes of irreducible $\BI$-modules that have dimension $d+1$. Let $\mathbf{EP}_d$ denote the set of all $(a,b,c)\in \F^3$ that satisfy 
$$
a+b+c, -a+b+c, a-b+c, a+b-c\not\in \left\{
\displaystyle{\frac{d-1}{2}-i}\,\bigg|\,i=0,2,\ldots,d-1
\right\}.
$$
Define an action of the abelian group $\{\pm 1\}^3$ on $\mathbf{EP}_d$ by 
\begin{align*}
(a,b,c)^{(-1,1,1)} &= (-a,b,c),
\\
(a,b,c)^{(1,-1,1)} &= (a,-b,c),
\\
(a,b,c)^{(1,1,-1)} &= (a,b,-c)
\end{align*}
for all $(a,b,c)\in \mathbf{EP}_d$. Let $\mathbf{EP}_d/\{\pm 1\}^3$ denote the set of the $\{\pm 1\}^3$-orbits of $\mathbf{EP}_d$. For $(a,b,c)\in \mathbf{EP}_d$, let $[a,b,c]$ denote the $\{\pm 1\}^3$-orbit of $\mathbf{EP}_d$ that contains $(a,b,c)$. Then there exists a bijection $\mathcal E:\{\pm 1\}^2\times \mathbf{EP}_d/\{\pm 1\}^3 \to \mathbf{EM}_d$ given by
\begin{eqnarray*}
((\e,\e'),[a,b,c])
&\mapsto &
\hbox{the isomorphism class of $E_d(a,b,c)^{(\e,\e')}$}
\end{eqnarray*}
for all $(\e,\e')\in \{\pm 1\}^2$ and all $[a,b,c]\in \mathbf{EP}_d/\{\pm 1\}^3$.
\end{thm}

We now turn our attention to the odd-dimensional irreducible $\BI$-modules.

\begin{prop}\label{prop:Od} 
For any scalars $a,b,c\in \F$ and any even integer $d\geq 0$, there exists a $(d+1)$-dimensional $\BI$-module $O_d(a,b,c)$ satisfying the following conditions:
\begin{enumerate}
\item There exists an $\mathbb F$-basis for $O_d(a,b,c)$ with respect to which the matrices representing $X$ and $Y$ are 
$$
\begin{pmatrix}
\theta_0 & & &  &{\bf 0}
\\ 
1 &\theta_1 
\\
&1 &\theta_2
 \\
& &\ddots &\ddots
 \\
{\bf 0} & & &1 &\theta_d
\end{pmatrix},
\qquad 
\begin{pmatrix}
\theta_0^* &\varphi_1 &  & &{\bf 0}
\\ 
 &\theta_1^* &\varphi_2
\\
 &  &\theta_2^* &\ddots
 \\
 & & &\ddots &\varphi_d
 \\
{\bf 0}  & & & &\theta_d^*
\end{pmatrix}
$$
respectively, where 
\begin{align*}
\theta_i &=\frac{(-1)^i(2a-d+2i)}{2}  \qquad (0\leq i\leq d),
\\
\theta^*_i &=\frac{(-1)^i(2b-d+2i)}{2} \qquad (0\leq i\leq d),
\\
\varphi_i &=
\left\{
\begin{array}{ll}
\displaystyle{\frac{i(d+1-2i-2a-2b-2c)}{2}}
\qquad &\hbox{if $i$ is even},
\\
\displaystyle{\frac{(i-d-1)(d+1-2i-2a-2b+2c)}{2}}
\qquad &\hbox{if $i$ is odd}
\end{array}
\right.
\qquad (1\leq i\leq d).
\end{align*}

\item The elements $\kappa,\lambda,\mu$ act on $O_d(a,b,c)$ as scalar multiplication by
\allowdisplaybreaks
\begin{align*}
2ab-c(d+1),
\\
2bc-a(d+1),
\\
2ca-b(d+1),
\end{align*}
respectively.
\end{enumerate}
\end{prop}
\begin{proof}
It is straightforward to verify the proposition using Lemma \ref{lem:BIpresentation}.
\end{proof}

The following result is concerning the $\BI$-modules $O_d(a,b,c)^{(\e,\e')}$ for all $(\e,\e')\in \{\pm 1\}^2$. To prove this, one may follow the proof of Theorem \ref{thm:iso}.

\begin{thm}
Suppose that the $\BI$-module $O_d(a,b,c)$ is irreducible. Then the following hold:
\begin{enumerate}
\item The $\BI$-module $O_d(a,b,c)^{(1,-1)}$ is isomorphic to $O_d(a,-b,-c)$. 

\item The $\BI$-module $O_d(a,b,c)^{(-1,1)}$ is isomorphic to $O_d(-a,b,-c)$. 

\item The $\BI$-module $O_d(a,b,c)^{(-1,-1)}$ is isomorphic to $O_d(-a,-b,c)$. 
\end{enumerate}
\end{thm}

The classification of odd-dimensional irreducible $\BI$-modules is as follows:

\begin{thm}\label{thm:odd}
Assume that $\F$ is algebraically closed with ${\rm char\,}\F=0$. 
Let $d\geq 0$ denote an even integer. 
Let $\mathbf{OM}_d$ denote the set of all isomorphism classes of irreducible $\BI$-modules that have dimension $d+1$. Let $\mathbf{OP}_d$ denote the set of all $(a,b,c)\in \F^3$ that satisfy 
$$
a+b+c, a-b-c, -a+b-c, -a-b+c\not\in 
\left\{\frac{d+1}{2}-i\,\bigg|\,i=2,4,\ldots,d\right\}.
$$
Then there exists a bijection $\mathcal O:\mathbf{OP}_d\to \mathbf{OM}_d$ given by
\begin{eqnarray*}
(a,b,c)
&\mapsto &
\hbox{the isomorphism class of $O_d(a,b,c)$}
\end{eqnarray*}
for all $(a,b,c)\in \mathbf{OP}_d$.
\end{thm}

The proofs for Theorems \ref{thm:even} and \ref{thm:odd} are similar and tedious. Thus the rest of this paper is devoted to the proof of Theorem \ref{thm:even} and the proof of Theorem \ref{thm:odd} is omitted.

\section{An infinite-dimensional $\BI$-module and its universal property}\label{s:Verma}

The following notations are used throughout the rest of this paper: 
We let  $a,b,c,\delta$ be any scalars in $\F$ and let $d\geq 1$ be an odd integer. 
We let $\{v_i\}_{i=0}^d$ denote the $\F$-basis for $E_d(a,b,c)$ from Proposition \ref{prop:Ed}(i). We adopt the following parameters associated with $a,b,c,\delta$:
\begin{align}
\theta_i &=\frac{(-1)^i(2a-\delta+2i)}{2}
\qquad \hbox{for all $i\in \Z$},
\label{theta}
\\
\theta^*_i &=\frac{(-1)^i(2b-\delta+2i)}{2}
\qquad \hbox{for all $i\in \Z$},
\label{vtheta}
\\
\phi_i &=
\left\{
\begin{array}{ll}
\displaystyle{i (\delta-i+1)}
\qquad &\hbox{if $i$ is even},
\\
\displaystyle{c^2-\frac{(2b-2a-\delta+2i-1)^2}{4}}
\qquad &\hbox{if $i$ is odd}
\end{array}
\right.
\qquad \hbox{for all $i\in \Z$},
\label{phi}
\\
\varphi_i &=
\left\{
\begin{array}{ll}
\displaystyle{i (\delta-i+1)}
\qquad &\hbox{if $i$ is even},
\\
\displaystyle{c^2-\frac{(2a+2b-\delta+2i-1)^2}{4}}
\qquad &\hbox{if $i$ is odd}
\end{array}
\right.
\qquad \hbox{for all $i\in \Z$},
\label{vphi}
\\
\omega &=c^2-a^2-b^2+\frac{(\delta+1)^2}{4},
\label{omega}
\\
\omega^* &=a^2-b^2-c^2+\frac{(\delta+1)^2}{4},
\label{omega*}
\\
\omega^\diamond &=b^2-c^2-a^2+\frac{(\delta+1)^2}{4}.
\label{omaged}
\end{align}

\begin{prop}\label{prop:Verma}
There exists a $\BI$-module $M_\delta(a,b,c)$ satisfying the following conditions: 
\begin{enumerate}
\item There exists an $\mathbb F$-basis $\{m_i\}_{i=0}^\infty$ for $M_\delta(a,b,c)$ with respect to which the matrices representing $X$ and $Y$ are 
$$
\begin{pmatrix}
\theta_0 &  & & &  &{\bf 0}
\\ 
1 &\theta_1 
\\
&1 &\theta_2
 \\
 & &\cdot &\cdot
 \\
&  & &\cdot &\cdot
 \\
 {\bf 0} & & & &\cdot &\cdot
\end{pmatrix},
\qquad 
\begin{pmatrix}
\theta_0^* &\varphi_1 & & &  &{\bf 0}
\\ 
 &\theta_1^* &\varphi_2
\\
 &  &\theta_2^* &\cdot
 \\
  & & &\cdot &\cdot
  \\
&    & & &\cdot &\cdot
  \\
  {\bf 0} & &   & & &\cdot 
\end{pmatrix}
$$
respectively.

\item The elements $\kappa,\lambda,\mu$ act on $M_\delta(a,b,c)$ as scalar multiplication by
$\omega,\omega^*,\omega^\diamond$
respectively.
\end{enumerate}
\end{prop}
\begin{proof}
It is routine to verify the proposition using Lemma \ref{lem:BIpresentation}.
\end{proof}

Throughout the rest of this paper we let $\{m_i\}_{i=0}^\infty$ denote the $\F$-basis for $M_\delta(a,b,c)$ from Proposition \ref{prop:Verma}(i).

\begin{lem}\label{lem:mi}
For any integers $i,j$ with $0\leq i\leq j$ the following equation holds:
$$
m_{j+1}=\prod\limits_{h=i}^j (X-\theta_h)m_i.
$$
\end{lem} 
\begin{proof}
Immediate from Proposition \ref{prop:Verma}(i).
\end{proof}

We shall give an alternative description for the $\BI$-module $M_\delta(a,b,c)$. To do this we begin with a Poincar\'{e}--Birkhoff--Witt basis for $\BI$.

\begin{lem}\label{lem:PBWbasis}
The elements 
$$
X^i Z^j Y^k
\mu^r \lambda^s \kappa^t
\qquad 
\hbox{for all $i,j,k,r,s,t\in \N$}
$$
are an $\F$-basis for $\BI$.
\end{lem}
\begin{proof}
Recall from \cite[Theorem 3.4]{Huang:R<BI} that 
\begin{gather}\label{PBW}
X^i Y^j Z^k
\kappa^r \lambda^s \mu^t
\qquad 
\hbox{for all $i,j,k,r,s,t\in \N$}
\end{gather}
form an $\F$-basis for $\BI$. By Definition \ref{defn:BI} there exists a unique $\F$-algebra automorphism of $\BI$ that sends $X,Y,Z,\kappa,\lambda,\mu$ to $X,Z,Y,\mu,\lambda,\kappa$ respectively. The lemma follows by applying the automorphism to (\ref{PBW}). 
\end{proof}

Let $I_\delta(a,b,c)$ denote the left ideal of $\BI$ generated by 
\begin{gather}
Y-\theta_0^*,
\label{I1}
\\
(Y-\theta_1^*)(X-\theta_0)-\varphi_1,
\label{I2}
\\
\kappa-\omega, 
\quad 
\lambda-\omega^*,
\quad 
\mu-\omega^\diamond.
\label{I3}
\end{gather}

\begin{lem}\label{lem:cosets}
For all $n\in \N$ the following hold:
\begin{enumerate}
\item $YX^n+I_\delta(a,b,c)$ is an $\F$-linear combination of $X^i+I_\delta(a,b,c)$ for all $0\leq i\leq n$.

\item $ZX^n+I_\delta(a,b,c)$ is an $\F$-linear combination of $X^i+I_\delta(a,b,c)$ for all $0\leq i\leq n+1$.

\item $Z^n+I_\delta(a,b,c)$ is an $\F$-linear combination of $X^i+I_\delta(a,b,c)$ for all $0\leq i\leq n$.
\end{enumerate}
\end{lem}
\begin{proof}
(i): Proceed by induction on $n$.
Since $I_\delta(a,b,c)$ contains (\ref{I1}) it is true when $n=0$. Since $I_\delta(a,b,c)$ contains (\ref{I1}) and (\ref{I2}) it is true when $n=1$.
Now suppose $n\geq 2$. 
Right multiplying either side of (\ref{XXY}) by $X^{n-2}$ yields that 
\begin{gather*}
X^2 Y X^{n-2}+2XYX^{n-1}+Y X^n-Y X^{n-2}
=2 X^{n-1}\kappa+X^{n-2}\mu.
\end{gather*}
Since $I_\delta(a,b,c)$ contains (\ref{I3}) it follows that $YX^n$ is congruent to 
\begin{gather}\label{YXn}
Y X^{n-2}-X^2 Y X^{n-2}-2XYX^{n-1}
+2\omega X^{n-1}+\omega^\diamond X^{n-2}
\end{gather}
modulo $I_\delta(a,b,c)$.
By induction hypothesis the term (\ref{YXn})
is congruent to an $\F$-linear combination of $X^i$ for all $0\leq i\leq n$ modulo $I_\delta(a,b,c)$. Therefore (i) follows.

(ii): Right multiplying either side of (\ref{kappa}) by $X^n$ yields that 
\begin{align}\label{ZXn+I}
ZX^n=XYX^n+YX^{n+1}-X^n \kappa 
\pmod{I_\delta(a,b,c)}.
\end{align}
Since $I_\delta(a,b,c)$ contains (\ref{I3}) and by (i), the right-hand side of (\ref{ZXn+I}) is congruent to an $\F$-linear combination of $X^i$ for all $0\leq i\leq n+1$ modulo $I_\delta(a,b,c)$. Therefore (ii) follows.

(iii): Proceed by induction on $n$. There is nothing to prove for $n=0$. Suppose $n\geq 1$. By induction hypothesis $Z^n$ is congruent to an $\F$-linear combination of 
\begin{gather}\label{ZXi+I}
ZX^i
\qquad 
\hbox{for all $0\leq i\leq n-1$}
\end{gather}
modulo $I_\delta(a,b,c)$. 
By (ii) each of (\ref{ZXi+I}) is congruent to an $\F$-linear combination of $X^k$ for all $0\leq k\leq n$ modulo $I_\delta(a,b,c)$. Therefore (iii) follows.
\end{proof}

\begin{lem}\label{lem:basis_BI/I}
The $\F$-vector space $\BI/I_\delta(a,b,c)$ is spanned by 
$$
X^i+I_\delta(a,b,c)
\qquad 
\hbox{for all $i\in \N$}.
$$
\end{lem}
\begin{proof}
By Lemma \ref{lem:PBWbasis} the elements 
$$
X^i Z^j Y^k
\mu^r \lambda^s \kappa^t+I_\delta(a,b,c)
\qquad 
\hbox{for all $i,j,k,r,s,t\in \N$}
$$
span $\BI/I_\delta(a,b,c)$.  Since $I_\delta(a,b,c)$ contains (\ref{I1}) and (\ref{I3}) it follows that 
\begin{gather}\label{XZcoset}
X^i Z^j+I_\delta(a,b,c)
\qquad 
\hbox{for all $i,j\in \N$}
\end{gather}
span $\BI/I_\delta(a,b,c)$. Applying Lemma \ref{lem:cosets}(iii) each of (\ref{XZcoset}) is an $\F$-linear combination of $X^k+I_\delta(a,b,c)$ for all $k\in \N$. The lemma follows.
\end{proof}

\begin{thm}\label{thm:universal_prop}
There exists a unique $\BI$-module homomorphism
$$
\Phi:\BI/I_\delta(a,b,c) 
\to M_\delta(a,b,c)
$$ 
that sends $1+I_\delta(a,b,c)$ to $m_0$. Moreover $\Phi$ is an isomorphism.
\end{thm}
\begin{proof}
Consider the $\BI$-module homomorphism $\Psi:\BI\to M_\delta(a,b,c)$ that sends $1$ to $m_0$. By Proposition \ref{prop:Verma}(i), (\ref{I1}) and (\ref{I2}) are in the kernel of $\Psi$. By Proposition \ref{prop:Verma}(ii) each of (\ref{I3}) is in the kernel of $\Psi$. It follows that $I_\delta(a,b,c)$ is contained in the kernel of $\Psi$. Hence $\Psi$ induces a  $\BI$-module homomorphism $\BI/I_\delta(a,b,c) 
\to M_\delta(a,b,c)$ that maps $1+I_\delta(a,b,c)$ to $m_0$. The existence of $\Phi$ follows. Since the $\BI$-module $\BI/I_\delta(a,b,c)$ is generated by $1+I_\delta(a,b,c)$ the uniqueness follows.

By Proposition \ref{prop:Verma}(i) the homomorphism $\Phi$ maps
\begin{eqnarray}\label{X-theta+I}
\prod_{h=0}^{i-1} (X-\theta_h)+I_\delta(a,b,c)
\end{eqnarray}
to $m_i$ for all $i\in \N$. Since the vectors $\{m_i\}_{i\in \N}$ are linearly independent, it follows that (\ref{X-theta+I}) for all $i\in \N$ are linearly independent. By Lemma \ref{lem:basis_BI/I} the cosets (\ref{X-theta+I}) for all $i\in \N$ span $\BI/I_\delta(a,b,c)$ and hence these cosets form an $\F$-basis for $\BI/I_\delta(a,b,c)$. Since $\Phi$ maps an $\F$-basis for $\BI/I_\delta(a,b,c)$ to an $\F$-basis for $M_\delta(a,b,c)$, it follows that $\Phi$ is an isomorphism.
\end{proof}

As a consequence of Theorem \ref{thm:universal_prop} the $\BI$-module $M_\delta(a,b,c)$ has the following universal property:

\begin{prop}\label{prop:universal_prop}
If $V$ is a $\BI$-module which contains a vector $v$ satisfying
\begin{gather*}
Yv=\theta_0^*v,
\\
(Y-\theta_1^*)(X-\theta_0)v=\varphi_1 v,
\\
\kappa v=\omega v, 
\quad 
\lambda v=\omega^* v,
\quad 
\mu v=\omega^\diamond v,
\end{gather*}
then there exists a unique $\BI$-module homomorphism $M_\delta(a,b,c)\to V$ that sends $m_0$ to $v$.
\end{prop}

From now on until the end of this paper, we set $\delta=d$. Let
$$
N_d (a,b,c)
$$
denote the $X$-cyclic $\F$-subspace of $M_d(a,b,c)$ generated by $m_{d+1}$.

\begin{lem}\label{lem:N}
$N_d(a,b,c)$ is a $\BI$-submodule of $M_d(a,b,c)$
with the $\F$-basis
$
\{m_i\}_{i=d+1}^\infty.
$
\end{lem}
\begin{proof}
Let $N$ denote the $\F$-subspace of $M_d(a,b,c)$ spanned by $\{m_i\}_{i=d+1}^\infty$. It follows from Lemma \ref{lem:mi} that 
$$
(X-\theta_i)m_i=m_{i+1}
\qquad 
\hbox{for all $i\geq d+1$}.
$$
Hence $N$ is an $X$-invariant $\F$-subspace of $N_{d}(a,b,c)$. It follows from the construction of $N_d(a,b,c)$ that $N=N_d(a,b,c)$. Therefore $N_d(a,b,c)$ has the $\F$-basis
$
\{m_i\}_{i=d+1}^\infty.
$
From Proposition \ref{prop:Verma}(i) we see that 
$$
(Y-\theta_i^*)m_i=\varphi_i m_{i-1}
\qquad 
\hbox{for all $i\geq d+1$}.
$$
By (\ref{vphi}) the scalar $\varphi_{d+1}=0$ under the setting $\delta=d$. Hence $N_d(a,b,c)$ is $Y$-invariant. By Proposition \ref{prop:Verma}(ii) the element $\kappa$ acts on $N_d(a,b,c)$ as scalar multiplication. 
Therefore $N_d(a,b,c)$ is a $\BI$-module by Lemma \ref{lem:XYkappa}. The lemma follows.
\end{proof}

\begin{lem}\label{lem:E}
There exists a unique $\BI$-module isomorphism 
$$
M_d(a,b,c)/N_d(a,b,c)\to E_d(a,b,c)
$$ 
that sends $m_i+N_d(a,b,c)$ to $v_i$ for all $0\leq i\leq d$.
\end{lem}
\begin{proof}
By Lemma \ref{lem:N} the quotient space $M_d(a,b,c)/N_d(a,b,c)$ is a $(d+1)$-dimensional $\BI$-module with the $\F$-basis 
\begin{gather}\label{basis:M/N}
\{m_i+N_d(a,b,c)\}_{i=0}^d.
\end{gather}
Comparing Propositions \ref{prop:Ed}(i) and \ref{prop:Verma}(i) the matrices representing $X$ and $Y$ with respect to the $\F$-basis $\{v_i\}_{i=0}^d$ for $E_d(a,b,c)$ are identical with the matrices representing $X$ and $Y$ with respect the $\F$-basis (\ref{basis:M/N}) for $M_d(a,b,c)/N_d(a,b,c)$, respectively.  By Propositions \ref{prop:Ed}(ii) and \ref{prop:Verma}(ii) the actions of $\kappa$ on $E_d(a,b,c)$ and $M_d(a,b,c)/N_d(a,b,c)$ are scalar multiplication by the same scalar $\omega$.
Therefore the lemma follows by Lemma \ref{lem:XYkappa}.
\end{proof}

\begin{prop}\label{prop:E}
If there is a $\BI$-module homomorphism $M_d(a,b,c)\to V$ that sends $m_0$ to $v$ and 
\begin{gather}\label{md+1_ker}
\prod_{i=0}^d (X-\theta_i) v=0,
\end{gather}
then there exists a $\BI$-module homomorphism $E_d(a,b,c)\to V$ that sends $v_0$ to $v$.
\end{prop}
\begin{proof}
Denote by $\varrho$ the $\BI$-module homomorphism $M_d(a,b,c)\to V$. It follows from Lemma \ref{lem:mi} that 
$$
m_{d+1}=\prod_{i=0}^d (X-\theta_i) m_0
$$
Combined with (\ref{md+1_ker}) this yields that $m_{d+1}$ lies in the kernel of $\varrho$. Hence $N_d(a,b,c)$ is contained in the kernel of $\varrho$. By Lemma \ref{lem:N} the homomorphism $\varrho$ induces a $\BI$-module homomorphism $M_d(a,b,c)/N_d(a,b,c)\to V$ that sends $m_0+N_d(a,b,c)$ to $v$. Combined with Lemma \ref{lem:E} the proposition follows. 
\end{proof}

\section{The conditions for the $\BI$-module $E_d(a,b,c)$ as irreducible}\label{s:irr}

The goal of this section is to establish the necessary and sufficient conditions for $E_d(a,b,c)$ to be irreducible in terms of the parameters $a,b,c,d$. In this section we set 
\begin{gather}\label{e:wi}
w_i=\prod_{h=0}^{i-1} (X-\theta_{d-h}) v_0 
\qquad 
(0\leq i\leq d).
\end{gather}

\begin{lem}\label{lem:irr1}
If the $\BI$-module $E_d(a,b,c)$ is irreducible then the following conditions hold:
\begin{enumerate}
\item ${\rm char\, }\F\centernot\mid i$ for all $i=2,4,\ldots,d-1$.

\item $a+b+c, a+b-c\not\in \left\{
\displaystyle{\frac{d-1}{2}-i}\,\bigg|\,i=0,2,\ldots,d-1
\right\}$.
\end{enumerate}
\end{lem}
\begin{proof}
Suppose that there is an integer $i$ with $1\leq i\leq d$ such that $\varphi_i=0$. By Proposition \ref{prop:Ed} the $\F$-subspace $W$ of $E_d(a,b,c)$
spanned by 
$
\{v_h\}_{h=i}^d
$
is invariant under $X,Y,\kappa$. It follows from Lemma \ref{lem:XYkappa} that $W$ is a $\BI$-submodule of $E_d(a,b,c)$, a contradiction to the irreducibility of $E_d(a,b,c)$. Therefore $\varphi_i\not=0$ for all $1\leq i\leq d$, which is equivalent to (i) and (ii) by (\ref{vphi}). The lemma follows.
\end{proof}

\begin{lem}\label{lem:iso2}
$\{w_i\}_{i=0}^d$ is an $\F$-basis for $E_d(a,b,c)$.
\end{lem}
\begin{proof}
 It follows from Proposition \ref{prop:Ed}(i) that 
$$
v_i=\prod_{h=0}^{i-1} (X-\theta_{h}) v_0
\qquad 
(0\leq i\leq d). 
$$
Comparing with (\ref{e:wi}) the lemma follows.
\end{proof}

\begin{prop}\label{prop:iso2}
The $\BI$-module $E_d(a,b,c)$ is isomorphic to $E_d(-a,b,c)$. 
Moreover the matrices representing $X$ and $Y$ with respect to the $\F$-basis $\{w_i\}_{i=0}^d$ for  $E_d(a,b,c)$ are 
\begin{gather}\label{XY_Ed(-a,b,c)}
\begin{pmatrix}
\theta_d & & &  &{\bf 0}
\\ 
1 &\theta_{d-1} 
\\
&1 &\theta_{d-2}
 \\
& &\ddots &\ddots
 \\
{\bf 0} & & &1 &\theta_0
\end{pmatrix},
\qquad 
\begin{pmatrix}
\theta_0^* &\phi_1 &  & &{\bf 0}
\\ 
 &\theta_1^* &\phi_2
\\
 &  &\theta_2^* &\ddots
 \\
 & & &\ddots &\phi_d
 \\
{\bf 0}  & & & &\theta_d^*
\end{pmatrix}
\end{gather}
respectively.
\end{prop}
\begin{proof}
By Proposition \ref{prop:Ed}(i) there exists an $\F$-basis $\{u_i\}_{i=0}^d$ for $E_d(-a,b,c)$ with respect to which the matrices representing $X$ and $Y$ are equal to the matrices (\ref{XY_Ed(-a,b,c)}). By Lemma \ref{lem:iso2} it suffices to show that there is a $\BI$-module homomorphism $E_d(a,b,c)\to E_d(-a,b,c)$ that sends $w_i$ to $u_i$ for all $0\leq i\leq d$.

Observe that $Y u_0=\theta_0^* u_0$ and a direct calculation yields that 
\begin{gather*}
(Y-\theta_1^*)(X-\theta_0) u_0
=\varphi_1 u_0.
\end{gather*}
By Proposition \ref{prop:Ed}(ii) the elements $\kappa,\lambda,\mu$ act on $E_d(-a,b,c)$ as scalar multiplication by $\omega,\omega^*,\omega^\diamond$ respectively. According to Proposition \ref{prop:universal_prop} there exists a unique $\BI$-module homomorphism $M_d(a,b,c)\to E_d(-a,b,c)$ that sends $m_0$ to $u_0$. 
Using (\ref{XY_Ed(-a,b,c)}) yields that
\begin{gather*}
\prod_{i=0}^d (X-\theta_i)u_0=0.
\end{gather*}
Hence there exists a $\BI$-module homomorphism 
\begin{gather}\label{E(a)->E(-a)}
E_d(a,b,c)\to E_d(-a,b,c)
\end{gather}
that maps $v_0$ to $u_0$ by Proposition \ref{prop:E}. Using (\ref{e:wi}) yields that (\ref{E(a)->E(-a)}) sends $w_i$ to $u_i$ for all $0\leq i\leq d$. The proposition follows.
\end{proof}

\begin{lem}\label{lem:irr2}
If the $\BI$-module $E_d(a,b,c)$ is irreducible then the following conditions hold:
\begin{enumerate}
\item ${\rm char\, } \F\centernot\mid i$ for all $i=2,4,\ldots,d-1$.

\item $a+b+c, -a+b+c, a-b+c, a+b-c\not\in \left\{
\displaystyle{\frac{d-1}{2}-i}\,\bigg|\,i=0,2,\ldots,d-1
\right\}$.
\end{enumerate}
\end{lem}
\begin{proof}
By Proposition \ref{prop:iso2} the $\BI$-modules $E_d(a,b,c)$ and $E_d(-a,b,c)$ are isomorphic. Hence the lemma follows by applying Lemma \ref{lem:irr1} to $E_d(a,b,c)$ and $E_d(-a,b,c)$.
\end{proof}

Consider the operators 
\begin{align*}
R&=\prod_{h=1}^{d} (Y-\theta^*_h),
\\
S_i&=\prod_{h=1}^{d-i}(X-\theta_{d-h+1})
\qquad (0\leq i\leq d).
\end{align*}
It follows from Proposition \ref{prop:Ed}(i) that $Rv$ is a scalar multiple of $v_0$ for all $v\in E_d(a,b,c)$. Thus, for any integers $i,j$ with $0\leq i,j\leq d$ there exists a unique $L_{ij}\in \F$ such that 
\begin{gather}\label{defn:Lij}
RS_i v_j=L_{ij} v_0.
\end{gather}
Using Proposition \ref{prop:Ed}(i) yields that 
\begin{align}
L_{ij}&=0
\qquad 
(0\leq i<j\leq d),
\label{lowertriangular}
\\
L_{ij}&=(\theta_{i}-\theta_{j-1}) L_{i,j-1}+L_{i-1,j-1}
\qquad 
(1\leq j\leq i \leq d).
\label{L:rr}
\end{align}
It follows from Proposition \ref{prop:iso2} that 
\begin{gather}\label{Li0}
L_{i0}=\prod_{h=1}^{i} (\theta_0^*-\theta_{d-h+1}^*) \prod_{h=1}^{d-i} \phi_h 
\qquad 
(0\leq i\leq d).
\end{gather}
Solving the recurrence relation (\ref{L:rr}) with the initial condition (\ref{Li0}) yields that 
\begin{gather}\label{Lij}
L_{ij}=
\left\{
\begin{array}{ll}
\displaystyle{\prod_{h=1}^{i-j} (\theta_0^*-\theta_{d-h+1}^*)
\prod_{h=1}^{d-i} \phi_h 
\prod_{h=1}^{\lceil\frac{j}{2}\rceil}
\varphi_{2h-1}
\prod_{h=1}^{\lfloor\frac{j}{2}\rfloor}
\varphi_{2(\lfloor\frac{i}{2}\rfloor-h+1)}
}
\quad 
&\hbox{if $i$ is odd or $j$ is even},
\\
0 \quad 
&\hbox{otherwise}
\end{array}
\right.
\end{gather}
for all $0\leq j\leq i\leq d$.

\begin{thm}\label{thm:irrE}
The $\BI$-module $E_d(a,b,c)$ is irreducible if and only if the following conditions hold:
\begin{enumerate}
\item ${\rm char\, } \F\centernot\mid i$ for all $i=2,4,\ldots,d-1$.

\item $a+b+c, -a+b+c, a-b+c, a+b-c\not\in \left\{
\displaystyle{\frac{d-1}{2}-i}\,\bigg|\,i=0,2,\ldots,d-1
\right\}$.
\end{enumerate}
\end{thm}
\begin{proof} 
$(\Rightarrow)$: Immediate from Lemma \ref{lem:irr2}.

$(\Leftarrow)$: 
Let $W$ denote any nonzero $\BI$-submodule of $E_d(a,b,c)$.
It suffices to show that $W=E_d(a,b,c)$. Pick a nonzero vector $w\in W$. Since $W$ is invariant under $X$ and $Y$ it follows that 
\begin{gather}\label{W}
R S_i w\in W
\qquad 
(0\leq i\leq d).
\end{gather}
Since $\{v_i\}_{i=0}^d$ is an $\F$-basis for $E_d(a,b,c)$, there are $a_i\in \F$ for $0\leq i\leq d$ such that 
$$
w=\sum_{i=0}^d a_i v_i.
$$ 
Using (\ref{defn:Lij}) yields that 
\begin{gather}\label{RSw}
RS_i w=\left(\sum_{j=0}^d L_{ij} a_j\right) v_0
\qquad (0\leq i\leq d).
\end{gather}

Recall the parameters $\{\phi_i\}_{i\in \Z}$ and $\{\varphi_i\}_{i\in \Z}$ from (\ref{phi}) and (\ref{vphi}). 
By the conditions (i) and (ii) the parameters $\varphi_i\not=0$ and $\phi_i\not=0$ for all $1\leq i\leq d$. 
Let $L$ denote the square matrix indexed by $0,1,\ldots,d$ with $(i,j)$-entry as $L_{ij}$ for all $0\leq i,j\leq d$. 
By (\ref{lowertriangular}) the $(d+1)\times (d+1)$ matrix $L$ is lower triangular. By (\ref{Lij}) the diagonal entries of $L$ are 
$$
L_{ii}=
\prod_{h=1}^{d-i} \phi_h 
\prod_{h=1}^{i}
\varphi_i\not=0
\qquad 
(0\leq i\leq d).
$$
Therefore the matrix $L$ is nonsingular. Since $w$ is nonzero, at least one of $\{a_i\}_{i=0}^d$ is nonzero. Hence there exists an integer $i$ with $0\leq i\leq d$ such that 
$$
\sum_{j=0}^d L_{ij} a_j\not=0.
$$
Combined with (\ref{W}) and (\ref{RSw}) this yields that $v_0\in W$. Since the $\BI$-module $E_d(a,b,c)$ is generated by $v_0$ it follows that $W=E_d(a,b,c)$. The result follows. 
\end{proof}

\section{The isomorphism class of the $\BI$-module $E_d(a,b,c)$}\label{s:iso_class}

In Proposition \ref{prop:iso2} we saw that the $\BI$-module $E_d(a,b,c)$ is isomorphic to $E_d(-a,b,c)$. 
In this section we study the isomorphism class of the $\BI$-module $E_d(a,b,c)$ in further detail.

\begin{prop}\label{prop:iso1}
The $\BI$-module $E_d(a,b,c)$ is isomorphic to $E_d(a,b,-c)$.
\end{prop}
\begin{proof}
By Proposition \ref{prop:Ed}(i) there are $\F$-bases for $E_d(a,b,c)$ and $E_d(a,b,-c)$ with respect to which the matrices representing $X$ and $Y$ are the same. By Proposition \ref{prop:Ed}(ii) the actions of $\kappa$ on $E_d(a,b,c)$ and $E_d(a,b,-c)$ are multiplication by the same scalar $\omega$. It follows from Lemma \ref{lem:XYkappa} that $E_d(a,b,c)$ is isomorphic to $E_d(a,b,-c)$. 
\end{proof}

\begin{prop}\label{prop:iso3}
If the $\BI$-module $E_d(a,b,c)$ is irreducible then the $\BI$-module $E_d(a,b,c)$ is isomorphic to $E_d(a,-b,c)$.
\end{prop}
\begin{proof}
By Proposition \ref{prop:Ed}(i) there exists an $\F$-basis $\{u_i\}_{i=0}^d$ for $E_d(a,-b,c)$ with respect to which the matrices representing $X$ and $Y$ are 
\begin{gather}\label{XYEd(a,-b,c)}
\begin{pmatrix}
\theta_0 & & &  &{\bf 0}
\\ 
1 &\theta_1 
\\
&1 &\theta_2
 \\
& &\ddots &\ddots
 \\
{\bf 0} & & &1 &\theta_d
\end{pmatrix},
\qquad 
\begin{pmatrix}
\theta_d^* &\phi_d &  & &{\bf 0}
\\ 
 &\theta_{d-1}^* &\phi_{d-1}
\\
 &  &\theta_{d-2}^* &\ddots
 \\
 & & &\ddots &\phi_1
 \\
{\bf 0}  & & & &\theta_0^*
\end{pmatrix}
\end{gather}
respectively. Since the $\BI$-module $E_d(a,b,c)$ is irreducible, it follows from Theorem \ref{thm:irrE} that $\phi_i\not=0$ for all $1\leq i\leq d$. Hence we may set 
$$
v=\sum_{i=0}^d 
\prod_{h=1}^i
\frac{\theta_0^*-\theta_{d-h+1}^*}{\phi_{d-h+1}} u_i.
$$
A direct calculation yields that $Yv=\theta_0^* v$ and 
\begin{gather*}
(Y-\theta_1^*)(X-\theta_0) v
=\varphi_1 v.
\end{gather*}
By Proposition \ref{prop:Ed}(ii) the elements $\kappa,\lambda,\mu$ act on $E_d(a,-b,c)$ as scalar multiplication by $\omega,\omega^*,\omega^\diamond$ 
respectively. According to Proposition \ref{prop:universal_prop} there exists a unique $\BI$-module homomorphism $M_d(a,b,c)\to E_d(a,-b,c)$ that sends $m_0$ to $v$.  
Using (\ref{XYEd(a,-b,c)}) yields that 
\begin{gather*}
\prod_{i=0}^d (X-\theta_i)v=0.
\end{gather*}
By Proposition \ref{prop:E} there exists a $\BI$-module homomorphism 
\begin{gather}\label{E(b)->E(-b)}
E_d(a,b,c)\to E_d(a,-b,c)
\end{gather}
that sends $v_0$ to $v$. 
Since the $\BI$-module $E_d(a,b,c)$ is irreducible it follows from Theorem \ref{thm:irrE} that the $\BI$-module $E_d(a,-b,c)$ is irreducible. Hence the homomorphism (\ref{E(b)->E(-b)}) is onto. Since both of $E_d(a,b,c)$ and $E_d(a,-b,c)$ are of dimension $d+1$, it follows that (\ref{E(b)->E(-b)}) is an isomorphism. The proposition follows.
\end{proof}

We end this section with a brief summary of Propositions \ref{prop:iso2}, \ref{prop:iso1} and \ref{prop:iso3}.

\begin{thm}\label{thm:iso}
If the $\BI$-module $E_d(a,b,c)$ is irreducible then the $\BI$-module $E_d(a,b,c)$ is isomorphic to $E_d(-a,b,c)$, $E_d(a,-b,c)$ and $E_d(a,b,-c)$.
\end{thm}

\section{Proof of Theorem \ref{thm:even}}\label{s:classification_even}

In this section we are devoted to the proof of Theorem \ref{thm:even}.

\begin{lem}\label{lem:Schur}
Assume that $\F$ is algebraically closed. If $V$ is a finite-dimensional irreducible $\BI$-module, then each central element of $\BI$ acts on $V$ as scalar multiplication.
\end{lem}
\begin{proof}
Immediate from Schur's lemma.
\end{proof}

\begin{lem}\label{lem:theta}
For any $i\in \Z$ the following hold:
\begin{enumerate}
\item $\theta_{i+1}+\theta_{i-1}=-2\theta_i$.

\item $\theta_{i+1}\theta_{i-1}=(\theta_i-1)(\theta_i+1)$.
\end{enumerate}
\end{lem}
\begin{proof}
It is routine to verify the lemma using (\ref{theta}).
\end{proof}

\begin{thm}\label{thm:surjective}
Assume that $\F$ is algebraically closed with ${\rm char\,}\F=0$. If $V$ is a $(d + 1)$-dimensional irreducible $\BI$-module, then there exist $a, b, c \in \F$ and
$(\e, \e') \in \{\pm 1\}^2$ such that the $\BI$-module $E_d(a, b, c)^{(\e,\e')}$ is isomorphic to $V$.
\end{thm}
\begin{proof}
Given any scalar $\alpha \in \F$ we define 
$$
\vartheta_i(\alpha)
=(-1)^i(\alpha+i)
\qquad
\hbox{for all $i\in \Z$}.
$$
Since $\F$ is algebraically closed and $V$ is finite-dimensional, there exists an eigenvalue $\alpha$ of $X$ in $V$. Since ${\rm char\,}\F=0$, for any distinct integers $i,j$ the scalars $\vartheta_i(\alpha)$ and $\vartheta_j(\alpha)$ are equal if and only if $i+j=-2\alpha$. Since there are at most $d+1$ distinct eigenvalues of $X$ in $V$, there exists an integer $j$ such that $\vartheta_j(\alpha)$ is an eigenvalue of $X$ but $\vartheta_{j-1}(\alpha)$ is not  an eigenvalue of $X$ in $V$. Set 
$$
\varepsilon=(-1)^j,
\qquad 
a=\alpha+j+\frac{d}{2}.
$$
Similarly, there are a scalar $\beta\in \F$ and an integer $k$ such that $\vartheta_k(\beta)$ is an eigenvalue of $Y$ but $\vartheta_{k-1}(\beta)$ is not an eigenvalue of $Y$ in $V$. We set 
$$
\varepsilon'=(-1)^k,
\qquad 
b=\beta+k+\frac{d}{2}.
$$
Under the settings we have 
\begin{eqnarray}
\varepsilon \theta_i &=& \vartheta_{i+j}(\alpha) 
\qquad 
\hbox{for all $i\in \Z$},
\label{setting:a&e}
\\
\varepsilon' \theta_i^* &=& \vartheta_{i+k}(\beta)
\qquad 
\hbox{for all $i\in \Z$}.
\label{setting:b&e'}
\end{eqnarray}
By Lemma \ref{lem:Schur} the element $\kappa$ acts on $V^{(\e,\e')}$ as scalar multiplication. Since $\F$ is algebraically closed there exists a scalar $c\in \F$ such that the action of $\kappa$ on $V^{(\e,\e')}$ is the scalar multiplication by 
$$
\omega=c^2-a^2-b^2+\frac{(d+1)^2}{4}.
$$
To prove the theorem, it suffices to show that there exists a $\BI$-module isomorphism from $E_d(a,b,c)$ into $V^{(\e,\e')}$.

 Given any element $S\in \BI$ and any $\theta\in \F$ we let 
$$
V^{(\e,\e')}_S(\theta)=\{v\in V^{(\e,\e')}\,|\, Sv=\theta v\}.
$$
Pick any $v\in V^{(\e,\e')}_X(\theta_0)$. 
Applying $v$ to (\ref{XXY}) yields that 
\begin{gather}\label{XXY_V(theta)}
(X^2+2\theta_0 X+\theta_0^2-1) Y v=(2\theta_0\omega+\mu)v.
\end{gather}
By Lemma \ref{lem:theta} the left-hand side of (\ref{XXY_V(theta)}) is equal to 
$$
(X-\theta_{-1})(X-\theta_1) Y v.
$$
Left multiplying either side of (\ref{XXY_V(theta)}) by $(X-\theta_0)$ we obtain that 
$$
(X-\theta_{-1})(X-\theta_0)(X-\theta_1) Y v=0.
$$
By Table \ref{pm1-action} and (\ref{setting:a&e}) the scalar $\theta_{-1}$ is not an eigenvalue of $X$ in $V^{(\e,\e')}$. It follows that 
$$
(X-\theta_0)(X-\theta_1) Y v=0.
$$ 
In other words 
$
(X-\theta_1) Y v\in V^{(\e,\e')}_X(\theta_0)$. 
This shows that $V^{(\e,\e')}_X(\theta_0)$ is invariant under $(X-\theta_1) Y$. Since $\F$ is algebraically closed there exists an eigenvector $u$ of $(X-\theta_1) Y$ in $V^{(\e,\e')}_X(\theta_0)$. Similarly, there exists an eigenvector $w$ of $(Y-\theta_1^*)X$ in $\in V^{(\e,\e')}_Y(\theta_0^*)$. Define 
\begin{eqnarray}
u_i&=&\prod_{h=0}^{i-1} (Y-\theta_h^*) u
\qquad 
\hbox{for all $i\in \N$},
\label{vi}
\\
w_i&=&\prod_{h=0}^{i-1} (X-\theta_h) w
\qquad 
\hbox{for all $i\in \N$}.
\label{wi}
\end{eqnarray}

We proceed by induction to show that 
\begin{gather}\label{claim}
(X-\theta_i)u_i\in {\rm span}_\F\{u_0,u_1,\ldots,u_{i-1}\}
\qquad 
\hbox{for all $i\in \N$}.
\end{gather}
Since $u$ is an eigenvector of $(X-\theta_1) Y$ in $V^{(\e,\e')}_X(\theta_0)$, the claim is true for $i=0,1$. Now suppose that $i\geq 2$. Applying $u_{i-2}$ to (\ref{YYX}) we obtain that 
\begin{gather}\label{YYXvi-2}
(Y^2 X +2 YX Y +X Y^2 - X 
-2\omega Y )u_{i-2}
=
\lambda u_{i-2}.
\end{gather}
By Lemma \ref{lem:Schur} the right-hand side of (\ref{YYXvi-2}) is a scalar multiple of $u_{i-2}$. Applying induction hypothesis and (\ref{vi}) the left-hand side of (\ref{YYXvi-2}) is equal to 
\begin{gather}\label{YYXvi-2:LH'}
(\theta_{i-2} +2\theta_{i-1}+X)u_i
\end{gather}
plus an $\F$-linear combination of $u_0,u_1,\ldots,u_{i-1}$. By Lemma  \ref{lem:theta}(i) the term (\ref{YYXvi-2:LH'}) is equal to $(X-\theta_i) u_i$. Combining the above comments, the result (\ref{claim}) follows. 
Next, we show that $\{u_i\}_{i=0}^d$ is an $\F$-basis for $V^{(\e,\e')}$. Suppose on  the contrary that there is an integer $j$ with $0\leq j\leq d-1$ such that $u_{j+1}$ is an $\F$-linear combination of $u_0,u_1,\ldots,u_j$.
Let $W$ denote the $\F$-subspace of $V$ spanned by $u_0,u_1,\ldots,u_j$. 
Observe that $W$ is $Y$-invariant by (\ref{vi}) and $X$-invariant by (\ref{claim}).  It follows from Lemma \ref{lem:XYkappa} that $W$ is a nonzero $\BI$-submodule of $V^{(\e,\e')}$. Since the $\BI$-module $V^{(\e,\e')}$ is irreducible this implies that $W=V^{(\e,\e')}$. However $W$ is of dimension less than or equal to $d$, which contradicts that the dimension of $V^{(\e,\e')}$ is $d+1$. Therefore $\{u_i\}_{i=0}^d$ is an $\F$-basis for $V^{(\e,\e')}$. 
By similar arguments we have  
\begin{gather}\label{claim'}
(Y-\theta_i^*) w_i\in {\rm span}_\F\{w_0,w_1,\ldots,w_{i-1}\}
\qquad 
\hbox{for all $i\in \N$}
\end{gather}
and 
$\{w_i\}_{i=0}^d$ is an $\F$-basis for $V^{(\e,\e')}$.

By (\ref{claim}) the matrix representing $X$ with respect to the $\F$-basis $\{u_i\}_{i=0}^d$ for $V^{(\e,\e')}$ is an upper triangular matrix in which the diagonal entries are $\{\theta_i\}_{i=0}^d$. 
Applying the Cayley-Hamilton theorem yields that 
\begin{gather}\label{CH}
\prod\limits_{i=0}^d (X-\theta_i) w_0=0.
\end{gather}
Hence $Xw_d=\theta_dw_d$ by (\ref{wi}) and the matrix representing $X$ with respect to the $\F$-basis $\{w_i\}_{i=0}^d$ for $V^{(\e,\e')}$ is  
\begin{gather*}
\begin{pmatrix}
\theta_0 & & &  &{\bf 0}
\\ 
1 &\theta_1 
\\
&1 &\theta_2
 \\
& &\ddots &\ddots
 \\
{\bf 0} & & &1 &\theta_d
\end{pmatrix}.
\end{gather*}
By (\ref{claim'}) the matrix representing $Y$ with respect to the $\F$-basis $\{w_i\}_{i=0}^d$ for $V^{(\e,\e')}$ is upper triangular with diagonal entries $\{\theta_i^*\}_{i=0}^d$ and let $\{\varphi_i'\}_{i=1}^d$ denote its superdiagonal entries as follows:
\begin{gather*}
\begin{pmatrix}
\theta_0^* &\varphi_1' &  & &{\bf *}
\\ 
 &\theta_1^* &\varphi_2'
\\
 &  &\theta_2^* &\ddots
 \\
 & & &\ddots &\varphi_d'
 \\
{\bf 0}  & & & &\theta_d^*
\end{pmatrix}.
\end{gather*}
Applying $w_{i-1}$ to either side of (\ref{XXY}) and comparing the coefficients of $w_i$ we obtain that 
\begin{gather}\label{varphi'}
\varphi_{i+1}'+2\varphi_i'+\varphi_{i-1}'
=2\omega-(3\theta_i+\theta_{i-1})\theta_i^*-(\theta_i+3\theta_{i-1})\theta_{i-1}^*
\qquad 
(1\leq i\leq d),
\end{gather}
where $\varphi_0'$ and $\varphi_{d+1}'$ are interpreted as zero. It is straightforward to verify that $\varphi_i'=\varphi_i$ for all $1\leq i\leq d$ satisfy the recurrence relation (\ref{varphi'}). Since ${\rm char\,}\F=0$ the corresponding homogeneous recurrence relation 
$$
\sigma_{i+1}+2\sigma_i+\sigma_{i-1}=0
\qquad 
(1\leq i\leq d)
$$
with the initial values $\sigma_0=0$ and $\sigma_{d+1}=0$ has the unique solution $\sigma_i=0$ for all $0\leq i\leq d+1$. Thus $\varphi_i'=\varphi_i$ for all $1\leq i\leq d$. So far we have seen that  
\begin{gather}
Yw_0=\theta_0^* w_0,
\label{even:U1} 
\\
(Y-\theta_1^*)(X-\theta_0)w_0=\varphi_1 w_0,
\label{even:U2}
\\
\kappa w_0=\omega w_0.
\label{even:U3}
\end{gather}

Applying $w_0$ to either side of (\ref{YYX}) and simplifying the resulting equation by (\ref{even:U2}), it yields that 
\begin{gather}\label{even:U4}
\lambda w_0=\omega^* w_0.
\end{gather}
Similarly we have $\mu u_0=\omega^\diamond u_0$. It follows from Lemma \ref{lem:Schur} that 
\begin{gather}\label{even:U5}
\mu w_0=\omega^\diamond w_0.
\end{gather}
In light of (\ref{even:U1})--(\ref{even:U5}), it follows from Proposition \ref{prop:universal_prop} that there exists a unique $\BI$-module homomorphism $M_d(a,b,c)\to V^{(\e,\e')}$ that sends $m_0$ to $w_0$. Combined with (\ref{CH}) there exists a $\BI$-module homomorphism 
\begin{gather}\label{E->V}
E_d(a,b,c)\to V^{(\e,\e')}
\end{gather}
that sends $v_0$ to $w_0$ by Proposition \ref{prop:E}. 
Since the $\BI$-module $V^{(\e,\e')}$ is irreducible the homomorphism (\ref{E->V}) is onto. Since both of $E_d(a,b,c)$ and $V^{(\e,\e')}$ are of dimension $d+1$ it follows that (\ref{E->V}) is an isomorphism. The result follows.
\end{proof}

\begin{lem}\label{lem:trace}
For any $(\e,\e')\in \{\pm 1\}^2$ the traces of $X$ and $Y$
on the $\BI$-module $E_d(a,b,c)^{(\e,\e')}$ are 
$$
-\e\cdot \frac{d+1}{2}, 
\qquad 
-\e'\cdot \frac{d+1}{2},
$$ 
respectively.
\end{lem}
\begin{proof}
It is straightforward to verify the lemma using Proposition \ref{prop:Ed}(i) and Table \ref{pm1-action}. 
\end{proof}

\begin{lem}\label{lem:central}
The elements $\kappa+\mu,\lambda+\kappa,\mu+\lambda$ act on the $\BI$-module $E_d(a,b,c)$ as scalar multiplication by 
\begin{gather*}
-2
\left(
a+\frac{d+1}{2}
\right)
\left(
a-\frac{d+1}{2}
\right),
\\
-2
\left(
b+\frac{d+1}{2}
\right)
\left(
b-\frac{d+1}{2}
\right),
\\
-2
\left(
c+\frac{d+1}{2}
\right)
\left(
c-\frac{d+1}{2}
\right),
\end{gather*}
respectively. 
\end{lem}
\begin{proof}
It is straightforward to verify the lemma using Proposition \ref{prop:Ed}(ii). 
\end{proof}

\noindent {\it Proof of Theorem \ref{thm:even}.}
By Theorems \ref{thm:irrE} and \ref{thm:iso} the function $\mathcal E$ is well-defined.
By Theorem \ref{thm:surjective} the function $\mathcal E$ is surjective.
Since any element of $\BI$ has the same trace and any central element of $\BI$ acts as the same scalar on the isomorphic finite-dimensional irreducible $\BI$-modules, it follows from Lemmas \ref{lem:trace} and \ref{lem:central} that $\mathcal E$ is injective. The result follows.
\hfill $\square$

\subsection*{Acknowledgements}

The research is supported by the Ministry of Science and Technology of Taiwan under the project MOST 106-2628-M-008-001-MY4.

\bibliographystyle{amsplain}
\bibliography{MP}
\end{document}